\documentclass{conm-p-l}

\usepackage{}

\newtheorem{theorem}{Theorem}[section]
\newtheorem{lemma}[theorem]{Lemma}

\theoremstyle{definition}
\newtheorem{definition}[theorem]{Definition}

\theoremstyle{remark}

\numberwithin{equation}{section}

\newcommand{\R}{{\mathbb{R}}}
\newcommand{\Z}{{\mathbb{Z}}}
\newcommand{\C}{{\mathbb{C}}}

\begin{document}

\title{Instabilities in kinetic theory and their relationship to the ergodic theorem}


\author{Jonathan Ben-Artzi}
\address{DAMTP, University of Cambridge}
\curraddr{}
\email{}
\thanks{}

\subjclass[2010]{Primary }

\date{today}

\begin{abstract}
A method for obtaining simple criteria for instabilities in kinetic theory is described and outlined, specifically  for the relativistic Vlasov-Maxwell system. An important ingredient of the method is an analysis of a parametrized set of averaging operators along trajectories. This leads to a connection with similar problems in ergodic theory. In particular the rate of convergence in the ergodic theorem is a common feature which is studied.
\end{abstract}

\maketitle

\tableofcontents

\section{Introduction and main results}
\subsection{Overview}
Over the last decade significant progress has been made in the mathematical analysis of linear and nonlinear stability in collisionless kinetic theory. This paper continues a well-established sequence of results concerned with \emph{linear instabilities} in which the self-adjoint form of the problem is exploited in order to reduce the search for unstable modes to the search of certain resonances, see for instance \cite{Lin2001,Lin2008,Guo2008a,Ben-Artzi2011}. More recently, there has been significant progress in the understanding of \emph{nonlinear stability/instability} by using a certain continuum of conservation laws, called Casimirs \cite{Lemou2011a,Lemou2011b}. Finally, perhaps the most famous recent result in this context is the rigorous verification of the phenomenon known as ``Landau damping'' \cite{Mouhot2009}.

We are interested in linear instabilities of the Vlasov-Maxwell system describing the evolution of plasmas that are assumed to be colliosnless (either due to being extremely dilute, or due to being extremely hot where the time scales are such that collisions can be neglected).

\subsection{The relativistic Vlasov-Maxwell system}
The Vlasov equation is
	\begin{equation}\label{eq:vlasov}
	\frac{\partial f^\pm}{\partial t}+\hat{v}\cdot\nabla_x f^\pm+\mathbf F^\pm\cdot\nabla_vf^\pm=0,
	\end{equation}
where the two functions $f^\pm=f^\pm(t,x,v)\geq0$ represent the density of positively and negatively charged particles, respectively, that at time $t\in[0,\infty)$ are located at the point $x\in \mathbb R^d$ and have momentum $v\in\mathbb R^d$. In addition, $\hat{v}=v/\sqrt{1+|v|^2}$ is the relativistic velocity (the speed of light is taken as $c=1$ for simplicity) and $\mathbf F^\pm=\mathbf F^\pm(t,x,v)$ is the Lorentz force, given by
	\begin{equation}\label{eq:lorentz}
	\mathbf F^\pm=\pm\left(\mathbf E+\mathbf E^{ext}+\hat{v}\times (\mathbf B+\mathbf B^{ext})\right)
	\end{equation}
and providing the coupling to Maxwell's equations of electromagnetism
	\begin{equation}\label{eq:maxwells-eq}
	\nabla\cdot{\mathbf E}=\rho,\quad
	\nabla\cdot{\mathbf B}=0,\quad
	\nabla\times{\mathbf E}=-\frac{\partial{\mathbf B}}{\partial t},\quad
	\nabla\times{\mathbf B}={\mathbf j}+\frac{\partial {\mathbf E}}{\partial t}.
	\end{equation}
In the above, $\mathbf E=\mathbf E(t,x)$ and $\mathbf B=\mathbf B(t,x)$ are the electric and magnetic fields due to the plasma itself and $\mathbf E^{ext}=\mathbf E^{ext}(t,x),\ \mathbf B^{ext}=\mathbf B^{ext}(t,x)$ are the externally induced fields.  In addition, $\rho=\rho(t,x)$ is the charge density and ${\mathbf j}={\mathbf j}(t,x)$ is the current density, which are defined as
	\begin{equation}\label{eq:densities}
	\rho=\int (f^+-f^-)\;{d}v\quad\text{and}\quad{\mathbf j}=\int \hat{v} (f^+-f^-)\;{d}v.
	\end{equation}
We note that we have taken all constants that typically appear in these equations (such as the particle masses) to be $1$ to keep notation simple. External fields only complicate notation, and will therefore be omitted in what follows. Since we are interested in instabilities of equilibria, we define our notion of instability:
\begin{definition}[Linear instability]
We say that a given equilibrium $f^{0,\pm}(x,v)$ is \emph{linearly unstable}, if the system linearized around it has a purely growing mode solution of the form
	\begin{equation}\label{eq:purely-growing}
	\left(e^{t/T}f^\pm(x,v),e^{t/T}\mathbf{E}(x),e^{t/T}\mathbf{B}(x)\right),
	\quad
	T>0.
	\end{equation}
\end{definition}

\subsection{The $1.5$-dimensional case}
We restrict our attention to a lower dimensional version of the RVM system where certain symmetries are assumed yet all the main physical ingredients of the problem are kept  intact. In particular, the following is the lowest dimensional setting that allows for a nontrivial magnetic field. Spatially, all quantities are assumed to only depend upon the $x_1$ variable, while in the momentum variable dependence is possible upon $v_1$ \emph{and} $v_2$. Hence we write $x$ instead of $x_1$ for brevity, and $v=(v_1,v_2)$. Correspondingly, the electro-magnetic fields have the form $\mathbf E=(E_1,E_2,0)$ and $\mathbf B=(0,0,B)$. The RVM system is transformed into the following system of scalar equations:
		\begin{subequations}\label{eq:rvm-1.5}
		\begin{align}
		&\partial_tf^\pm+\hat{v}_1\partial_xf^\pm\pm(E_1+\hat{v}_2B)\partial_{v_1}f^\pm\pm(E_2-\hat{v}_1B)\partial_{v_2}f^\pm
		=
		0						\label{eq:vlasov-1.5}\\
		&\partial_tE_1
		=
		-j_1						\label{eq:ampere1-1.5}\\
		&\partial_tE_2+\partial_xB
		=
		-j_2						\label{eq:ampere2-1.5}\\
		&\partial_tB
		=
		-\partial_xE_2					\label{eq:faraday-1.5}\\
		&\partial_xE_1
		=
		\rho.						\label{eq:gauss-1.5}
		\end{align}
		\end{subequations}
For simplicity, throughout this paper we shall assume periodicity in the $x$ variable, with period $P$.  Next, we state our main instability result. The main significance of this result is in providing a relatively simple criterion for checking for linear instability: one only has to have knowledge of the spectra of certain Schr\"odinger operators acting on the spatial variable, not the entire phase-space variables. The theorem uses several definitions that are too technical to specify here, and therefore only appear later. 
\begin{theorem}\label{thm:main-intro}
Let $f^{0,\pm}(x,v)$ be an equilibrium of the $1.5$-dimensional RVM system \eqref{eq:rvm-1.5} satisfying the integrability condition \eqref{eq:int-condition} (see below). There exist two self-adjoint Schr\"odinger operators $\mathcal A^\infty_1$ and $\mathcal A^\infty_2$, a bounded operator $\mathcal B^\infty$ and a number $l^\infty$ (all defined below in \eqref{eq:infinity-operators}) acting only in the spatial variable (and not the momentum variable) such that the equilibrium is linearly unstable if

(i) only the constant functions are in $\ker\mathcal A_1^\infty$

(ii) the following inequality holds:
	\begin{equation}\label{eq:thm1-condition}
	\operatorname{neg}\left(\mathcal A_2^\infty+\mathcal B^\infty\left(\mathcal A_1^\infty\right)^{-1}\left(\mathcal B^\infty\right)^*\right)> \operatorname{neg}\left(\mathcal A_1^\infty\right)+\operatorname{neg}(-l^\infty),
	\end{equation}
where $\operatorname{neg}(A)$ is the number of negative eigenvalues of the self-adjoint operator $A$, and $\operatorname{neg}(-l^\infty)=1$ (resp. $0$) if $l^\infty>0$ (resp. $l^\infty\leq0$).
\end{theorem}

\subsection{The equilibrium}
Nontrivial equilibria $\left(f^{0,\pm}(x,v),E_1^0(x),E_2^0(x),B^0(x)\right)$ of \eqref{eq:rvm-1.5} are guaranteed to exist due to \cite{Glassey1990}. We shall always assume that $f^{0,\pm}\geq0$ are continuously differentiable and that $E_2^0\equiv0$. Define the energy $e^\pm$ and the momentum $p^\pm$, respectively, by the expressions
	\begin{equation}\label{eq:constants-1.5}
	e^\pm
	=
	\left<v\right>\pm\phi^0(x),\qquad
	p^\pm
	=
	v_2\pm\psi^0(x),
	\end{equation}
where $\phi^0$ and $\psi^0$ are the equilibrium electric and magnetic potentials, satisfying
	\begin{equation}\label{eq:potentials-1.5}
	\partial_x\phi^0
	=
	-E_1^0,
	\qquad
	\partial_x\psi^0
	=
	B^0.
	\end{equation}
It is well-known that $e^\pm,\ p^\pm$ are conserved along trajectories of the linearized Vlasov operators
	\begin{equation}\label{eq:trajectories}
	D^\pm
	=
	\hat{v}\cdot\nabla_x+\mathbf{F}^{0,\pm}\cdot\nabla_v=\left(\hat{v},\mathbf F^{0,\pm}\right)\cdot\nabla_{x,v}.
	\end{equation}
Using Jeans' theorem \cite{Jeans1915} we replace the coordinates $(x,v)$ by $(e^\pm,p^\pm)$ and write the equilibrium distributions as
	\begin{equation}\label{eq:jeans}
	f^{0,\pm}(x,v)=\mu^\pm(e^\pm,p^\pm).
	\end{equation}
We assume that there exist weight functions $w^\pm=c(1+|e^\pm|)^{-\alpha}$ with $\alpha>2$ and $c>0$ such that the integrability condition
	\begin{equation}\label{eq:int-condition}
	\left(\left|\frac{\partial\mu^\pm}{\partial e}\right|+\left|\frac{\partial\mu^\pm}{\partial p}\right|\right)(e^\pm,p^\pm)<w^\pm(e^\pm)
	\end{equation}
holds. This ensures that $\int\left(|\mu^\pm_e|+|\mu^\pm_p|\right)dxdv<\infty$, where we have used shorthand notation for the partial derivatives with respect to the first and second variables respectively (this notation shall appear throughout this paper). Moreover, we define the following functional spaces that include functions that do not necessarily decay
	\begin{equation*}
	L_\pm^2=\left\{h(x,v)\;\Bigg{|}\;  h \text{ is } P\text{-periodic in }x, \left\|h\right\|_\pm^2:=\int_0^P\int_{\mathbb{R}^2}|h|^2w^\pm\;dv\;dx<\infty\right\}.
	\end{equation*}
The norm and inner-product in $L^2_\pm$ are denoted $\|\cdot\|_{L^2_\pm}$ and $\left<\cdot,\cdot\right>_{L^2_\pm}$, respectively. In addition we define
	\begin{equation*}
	L_P^2=\left\{h\text{ is $P$-periodic and square integrable on }[0,P]\right\}
	\end{equation*}
as well as the spaces $L_{P,0}^2$ of functions in $L_P^2$ of zero mean value (over a period), $H_P^2$ of functions whose first and second distributional derivatives are in $L^2_P$ and $H_{P,0}^2$ of functions in $H_P^2$ of zero mean value. The norm and inner-product in $L^2_P$ are denoted $\|\cdot\|_{L^2_P}$ and $\left<\cdot,\cdot\right>_{L^2_P}$, respectively.

\subsection{Seeking a uniform ergodic theorem}
In the course of the proof of Theorem \ref{thm:main-intro} we encounter an application of the ergodic theorem where a detailed knowledge of convergence rates could be advantageous. This can be seen in the definition of the ergodic averaging operators $Q^T_\pm$ (Equation \eqref{eq:qt} below) whose properties are discussed in Lemma \ref{propql} below. These averages are taken along the trajectories of $D^\pm$.   Since a uniform rate of convergence does not exist in general, we present methods for obtaining such rates on certain subspaces.

Since the trajectories of $D^\pm$ (which represent the flow of the linearized RVM system) are quite complicated, we begin by studying first the simplest case of a $1D$ flow on $L^2$ and weighted-$L^2$ spaces. Generalizations of these results to higher dimensional shear flows  \cite{Ben-Artzi2013d} and more general flows \cite{Ben-Artzi2012b} are in preparation.

\emph{1. $L^2$ case.} We consider the self-adjoint operator
	\begin{equation}
	H=-i\frac{d}{dx}:H^1(\R)\subset L^2(\R)\to L^2(\R)
	\end{equation}
and we define the space $L^{2,\sigma}(\R)$ as
	\begin{equation}\label{eq:l2s}
	L^{2,\sigma}(\R)=\left\{f:\mathbb{C}\to\R\ \Big|\ \|f\|^2_{L^{2,\sigma}(\R)}:=\int(1+x^2)^\sigma |f(x)|^2\ dx<\infty\right\}.
	\end{equation}
Then we have:
\begin{theorem}[Uniform ergodic theorem -- $L^2$ case]\label{thm:unif-erg-l2}
For $\sigma>\frac12$, the self-adjoint operator $H=-i\frac{d}{dx}$ satisfies
	\begin{equation}
	\lim_{T\to\infty}\frac 1{2T}\int_{-T}^T e^{itH}dt=0
	\end{equation}
in the uniform operator topology on $\mathcal B(L^{2,\sigma}(\R),L^{2,-\sigma}(\R))$.
\end{theorem}

\emph{2. Weighted-$L^2$ case.} In the weighted case the result is more interesting, as the constant functions are part of our functional space. We let $0<w\in L^1(\R)\cap L^\infty(\R)$ and define the weighted space
	\begin{equation}
	L_w^2(\R)=\left\{f:\mathbb{C}\to\R\ \Big|\ \|f\|^2_{L^{2}_w(\R)}:=\int |f|^2w<\infty\right\}.
	\end{equation}
On this space, the operator $H$ is no longer symmetric. However, the operator
	\begin{equation}
	H_w=-\frac{i}{w}\frac{d}{dx}:L^2_w(\R)\to L_w^2(\R)
	\end{equation}
is. Self-adjointness, however, is less straightforward. Therefore we shall first prove:

\begin{theorem}\label{thm:self-adjoint}
The operator $H_w:D^\alpha\subset L^2_w(\R)\to L^2_w(\R)$ is essentially self-adjoint, where for a fixed $\alpha\in\mathbb{C}$ with $|\alpha|=1$
	\[
	D^\alpha=\left\{f\in L^2_w(\R)\ \Big| \ H_w f\in L^2_w(\R),\ \lim_{x\to\infty}f(x)=\alpha\lim_{x\to-\infty}f(x)\right\}.
	\]
We designate as $H_w^\alpha$ its unique self-adjoint extension.
\end{theorem}

With this statement at hand, we can now present a new uniform ergodic theorem:
\begin{theorem}[Uniform ergodic theorem -- weighted-$L^2$ case]\label{thm:unif-erg-wl2}
The convergence
	\[
	\lim_{T\to\infty}\frac 1{2T}\int_{-T}^T e^{itH_w^\alpha}\ dt=P
	\]
holds in the uniform operator topology on $\mathcal B(L^2_w(\R),L^2_w(\R))$ for each $|\alpha|=1$, where $P$ is the orthogonal projection onto the kernel of $H_w^\alpha$.
\end{theorem}

\subsection{Outline of the paper}
Our strategy for proving Theorem \ref{thm:main-intro} is to first make the ansatz that the linearized Vlasov equation is linearly unstable, as defined in \eqref{eq:purely-growing}. This provides us with a family of equations for $f^\pm$ depending upon the parameter $T>0$. This allows us to find expressions for $f^\pm$ (again, depending on $T$) which we substitute into Maxwell's equations via the charge and current densities. We thus obtain a one-parameter family of self-adjoint systems of equations (in the spatial variable only) depending upon the parameter $T$. We must show that this family has a solution for some $T>0$. This is done in Section \ref{sec:self-adjoint}. In Section \ref{sec:tracking} we use the self-adjointness of this family, as well as continuity properties of its spectrum (if such exist) in order to track its eigenvalues as $T$ varies from $0$ to $+\infty$. Under the conditions of Theorem \ref{thm:main-intro} we find an eigenvalue crossing through $0$, which justifies the ansatz. The most difficult step in the proof, which is merely sketched in the form of Theorem \ref{thm:resolvent}, is closely related to the ergodic theorem, and, specifically, to the non-existence of a rate of convergence for ergodic averages. In Section \ref{sec:ergodic} we consider this problem, and exhibit two instances where a uniform ergodic theorem (that is, an ergodic theorem \emph{with a rate}) can be shown to hold. This presents a first step towards more robust results in this direction. Finally, in Section \ref{sec:prop-oper} we gather some technical lemmas.

\section{An equivalent self-adjoint problem}\label{sec:self-adjoint}
\subsection{Reformulation of the problem}
Linearizing the Vlasov equation \eqref{eq:vlasov} we obtain
	\begin{equation}\label{eq:lin-vlasov}
	\frac{\partial f^\pm}{\partial t}+\hat{v}\cdot\nabla_x f^\pm+\mathbf F^{0,\pm}\cdot\nabla_vf^\pm=-\mathbf F^{\pm}\cdot\nabla_vf^{0,\pm},
	\end{equation}
which, after making the ansatz that the time dependence is as in \eqref{eq:purely-growing}, becomes
	\begin{equation}\label{eq:lin-vlasov2}
	\frac{1}{T}f^\pm+\hat{v}\cdot\nabla_x f^\pm+\mathbf F^{0,\pm}\cdot\nabla_vf^\pm=-\mathbf F^{\pm}\cdot\nabla_vf^{0,\pm}.
	\end{equation}
The right hand side of this equation includes the perturbed Lorentz forces $\mathbf{F}^\pm=\pm\left(\mathbf E+\hat{v}\times \mathbf B\right)$ which we want to express in terms of the electromagnetic potentials $\phi$ and $\psi$. Using Maxwell's equations, and the ansatz \eqref{eq:purely-growing} to replace time derivatives by $\frac{1}{T}$, we obtain the expressions
	\[
	B=\partial_x\psi,\quad E_2=-\frac{1}{T}\psi,\quad E_1=-\partial_x\phi-\frac{1}{T}b.
	\]
Above, $b\in\mathbb{R}$ is simply the mean value of $E_1$ over a period, and is an artifact due to the periodicity we have introduced. Hence \eqref{eq:lin-vlasov2} becomes
	\begin{equation}\label{eq:lin-vlasov3}
	\left(\frac{1}{T}+D^\pm\right)f^\pm=\pm\mu^\pm_e\hat{v}_1\left(\partial_x\phi+\frac{1}{T} b\right)\pm\mu^\pm_p\hat{v}_1\partial_x\psi\pm\frac{1}{T}\left(\mu^\pm_e\hat{v}_2+\mu^\pm_p\right)\psi
	\end{equation}
where the operators $D^\pm$ are given in \eqref{eq:trajectories}.
There are two parallel approaches for inverting this equation in order to obtain an expression for $f^\pm$. In the first, which can be found in \cite{Ben-Artzi2011}, we integrate \eqref{eq:lin-vlasov3} along the trajectories $\left(X^\pm(s;x,v),V^\pm(s;x,v)\right)$ of the vectorfields $D^\pm$ in phase space,  which satisfy
	\begin{align*}
	\dot{X}^\pm&=\hat{V}^\pm_1,\\
	\dot{V}^\pm_1&=
	\pm E_1^0\pm\hat{V}^\pm_2B^0(X^\pm),\\
	\dot{V}^\pm_2&=\mp\hat{V}^\pm_1B^0(X^\pm),
	\end{align*}
with the initial conditions
	\begin{equation*}
	\left(X^\pm(0;x,v),V^\pm(0;x,v)\right)
	=
	(x,v).
	\end{equation*}
Another point of view, presented in \cite{Ben-Artzi2013a}, is to apply the resolvents of $D^\pm$, which are skew-adjoint differential operators. In both cases, the expressions obtained are
	\begin{equation}\label{fpm}
	f^\pm(x,v)=\pm\mu^\pm_e\phi(x)\pm\mu^\pm_p\psi(x)\mp\mu^\pm_eQ^T_\pm\left(\phi-\hat{v}_2\psi-b\hat{v}_1\right)
	\end{equation}
where $Q^T_\pm:L^2_\pm\to L^2_\pm$ are the ergodic averaging operators given by, for $k\in L^2_\pm$,
	\begin{equation}\label{eq:qt}
	(Q^T_\pm k)(x,v)=\frac{1}{T}\int_{-\infty}^0 e^{{s}/{T}}k\left(X^\pm(s;x,v),V^\pm(s;x,v)\right) ds.
	\end{equation}
Substituting \eqref{fpm} into Maxwell's equations we obtain the self-adjoint system of equations
	\begin{equation}\label{eq:matrix1}
	\begin{split}
	&-\mathcal{A}_1^T\phi+\mathcal{B}^T\psi+\mathcal{C}^Tb=0\\
	&\left(\mathcal{B}^T\right)^*\phi+\mathcal{A}_2^T\psi-\mathcal{D}^Tb=0\\
	&\left(\mathcal{C}^T\right)^*\phi-\left(\mathcal{D}^T\right)^*\psi-P\left(T^{-2}-l^T\right)b=0
	\end{split}
	\end{equation}
where
	\begin{align*}
	\mathcal{A}_1^T h&=-\partial_x^2h-\left(\sum_\pm\int\mu^\pm_e\;dv\right)h+\sum_\pm\int\mu^\pm_eQ^T_\pm h\;dv,\\
	\mathcal{A}_2^T h&=-\partial_x^2h+T^{-2}h-\left(\sum_\pm\int\hat{v}_2\mu^\pm_p\;dv\right)h-\sum_\pm\int\mu^\pm_e\hat{v}_2Q^T_\pm (\hat{v}_2h)\;dv,\\
	\mathcal{B}^T h&=\left(\sum_\pm\int\mu^\pm_p\;dv\right)h+\sum_\pm\int\mu^\pm_eQ^T_\pm(\hat{v}_2h)\;dv,\\
	\mathcal{C}^Tb&=b\sum_\pm\int\mu^\pm_eQ^T_\pm\left(\hat{v}_1\right)dv,\\
	\mathcal{D}^Tb&=b\sum_\pm\int\hat{v}_2\mu^\pm_eQ^T_\pm\left(\hat{v}_1\right)dv,\\
	l^T&=\frac{1}{P}\sum_\pm\int_0^P\int\hat{v}_1\mu^\pm_eQ^T_\pm\left(\hat{v}_1\right)dv\;dx.
	\end{align*}
The detailed derivation of this system from Maxwell's equations may be found in \cite{Ben-Artzi2011}. A nontrivial solution of \eqref{eq:matrix1} for some $0<T<+\infty$ verifies the ansatz \eqref{eq:purely-growing} and implies the existence of a growing mode. For brevity we write \eqref{eq:matrix1} as
	\begin{equation}\label{eq:matrix2}
	\mathcal{M}^T\left(\begin{array}{l}\phi\\\psi\\b\end{array}\right)=0.
	\end{equation}
The properties of all operators introduced here shall be collectively discussed in Section \ref{sec:prop-oper}. Primarily, we care about Lemma \ref{propml} describing the properties of $\mathcal{M}^T$. Our strategy is to prove Theorem \ref{thm:main-intro} by showing that \eqref{eq:matrix2} has a nontrivial solution for some $0<T<+\infty$ by tracking the spectrum of $\mathcal{M}^T$ as $T$ varies from $0$ to $+\infty$, and seeking an eigenvalue that crosses through $0$. For this, the self-adjointness of $\mathcal{M}^T$ for all $T\geq0$ is crucial. In addition, we need to understand the form of $\mathcal{M}^T$ (and therefore all other operators) for $T=+\infty$. In particular, the ergodic averages $Q^T_\pm$ converge strongly,  as $T\to+\infty$,  to the projection operators $Q_\pm^\infty$ defined as (see Lemma \ref{propql} below for a precise statement):

\begin{definition}\label{proj-define}
We define $Q^\infty_\pm$ to be the orthogonal projections of $L_\pm^2$ onto $\ker D^\pm$.
\end{definition}
Accordingly, we define
	\begin{equation}\label{eq:infinity-operators}
	\begin{split}
	\mathcal{A}_1^\infty h&=-\partial_x^2h-\left(\sum_\pm\int\mu^\pm_e\;dv\right)h+\sum_\pm\int\mu^\pm_eQ^\infty_\pm h\;dv,\\
	\mathcal{A}_2^\infty h&=-\partial_x^2h+T^{-2}h-\left(\sum_\pm\int\hat{v}_2\mu^\pm_p\;dv\right)h-\sum_\pm\int\mu^\pm_e\hat{v}_2Q^\infty_\pm (\hat{v}_2h)\;dv,\\
	\mathcal{B}^\infty h&=\left(\sum_\pm\int\mu^\pm_p\;dv\right)h+\sum_\pm\int\mu^\pm_eQ^\infty_\pm(\hat{v}_2h)\;dv,\\
	l^\infty&=\frac{1}{P}\sum_\pm\int_0^P\int\hat{v}_1\mu^\pm_eQ^\infty_\pm\left(\hat{v}_1\right)dv\;dx.
	\end{split}
	\end{equation}

\section{Tracking the spectrum of $\mathcal{M}^T$ and finding a growing mode}\label{sec:tracking}
As described above, our strategy for finding a value of $0<T<+\infty$ for which $\mathcal{M}^T$ has a nontrivial kernel is to ``track'' its spectrum as $T$ varies from $0$ to $+\infty$ and find an eigenvalue crossing through $0$. As we state precisely in Lemma \ref{propml}, $(0,\infty)\ni T\mapsto\mathcal{M}^T$ is continuous in the uniform operator topology. However, the limit $\lim_{T\to\infty}\mathcal{M}^T=\mathcal{M}^\infty$ is only guaranteed to exist in the  strong topology. A major obstacle is that \emph{strong continuity \textbf{does not} guarantee continuity of the spectrum as a set} \cite[VIII-\S1.3]{Kato1995}. We therefore proceed by first solving an approximate problem set in a finite-dimensional subspace.

\subsection{The truncation}
We define the following two families of finite-dimensional orthogonal projection operators:
	\begin{equation}
	\begin{split}
	P_n&=\text{the orthogonal projection onto the eigenspace associated}\\&\qquad\text{with the first $n$ eigenvalues (counting multiplicity) of $\mathcal{A}_1^\infty$},
	\end{split}
	\end{equation}
and
	\begin{equation}
	\begin{split}
	Q_n&=\text{the orthogonal projection onto the eigenspace associated}\\&\qquad\text{with the first $n$ eigenvalues (counting multiplicity) of $\mathcal{A}_2^\infty$.}
	\end{split}
	\end{equation}
Then we define the \emph{truncated matrix operator} to be
	\begin{equation}\label{def:trunc-op}
	\mathcal{M}^T_n
	=
	\left(
	\begin{array}{ccc}
	{-\mathcal{A}^T_{1,n}}&{\mathcal{B}^T_n}&{\mathcal{C}^T_n}\\
	{\left(\mathcal{B}^T_n\right)^*}&{\mathcal{A}^T_{2,n}}&{-\mathcal{D}^T_n}\\
	{\left(\mathcal{C}^T_n\right)^*}&{-\left(\mathcal{D}^T_n\right)^*}&{-P\left(T^{-2}-l^T\right)}\end{array}
	\right)
	\end{equation}
where
	\begin{align*}
	\mathcal{A}^T_{1,n}=P_n\mathcal{A}^T_1 P_n\quad\qquad\qquad\qquad\mathcal{A}^T_{2,n}=Q_n\mathcal{A}_2^T Q_n\\
	\mathcal{B}^T_n=P_n\mathcal{B}^T Q_n  \qquad  \mathcal{C}^T_n=P_n\mathcal{C}^T\qquad\mathcal{D}^T_n=Q_n\mathcal{D}^T.
	\end{align*}
When $T=+\infty$ the truncated matrix operator becomes
	\begin{equation}\label{eq:mn-infty}
	\mathcal{M}^\infty_n
	=
	\left(
	\begin{array}{ccc}
	{-\mathcal{A}^\infty_{1,n}}&{\mathcal{B}^\infty_n}&{0}\\
	{\left(\mathcal{B}^\infty_n\right)^*}&{\mathcal{A}^\infty_{2,n}}&{0}\\
	{0}&{0}&{Pl^\infty}\end{array}
	\right)
	\end{equation}

\subsection{$T$ small}
For small values of $T$ the analysis is rather simple due to the appearance of the terms $T^{-2}$. We have
\begin{lemma}\label{large-la}
There exists $T_*>0$ such that for any $n\in\mathbb{N}$ and any $T<T_*$, $\mathcal{M}^T_n$ has exactly $n+1$ negative eigenvalues.
\end{lemma}

\begin{proof}
Since $\mathcal{M}_n^T$ is a symmetric mapping on a  $2n+1$-dimensional subspace of $H^2_{P,0}\times H^2_{P}\times\R$ it has $2n+1$ real eigenvalues. Letting $\psi\in H^2_{P}$, we have

	\begin{equation}
	\left<\mathcal{M}_n^T \left(\begin{array}{c}0\\\psi\\0\end{array}\right),\left(\begin{array}{c}0\\\psi\\0\end{array}\right)\right>_{L^2_P\times L^2_P\times\R}
	=
	\left<\mathcal{A}_2^T Q_n\psi,Q_n\psi\right>_{L^2_P}
	>
	0
	\end{equation}
for all $T<\overline{T}$ by Lemma \ref{properties1}. This implies that $\mathcal{M}_n^T$ is positive definite on a subspace of dimension $n$, and, therefore it has at least $n$ positive eigenvalues. Similarly, we now show that there exists a subspace of dimension $n+1$ on which $\mathcal{M}_n^T$ is negative definite. Let $(\phi,0,b)\in H^2_{P,0}\times H^2_{P}\times\R$ and consider
	\begin{align}\label{neg-def}
	\left<\mathcal{M}_n^T \left(\begin{array}{c}\phi\\0\\b\end{array}\right),\left(\begin{array}{c}\phi\\0\\b\end{array}\right)\right>_{L^2_P\times L^2_P\times\R}
	&=
	-\left<\mathcal{A}_1^T P_n\phi,P_n\phi\right>_{L^2_P}
	+
	2\left<\mathcal{C}^T b,P_n\phi\right>_{L^2_P}
	-
	P(T^{-2}-l^T)b^2.
	\end{align}
We estimate the second term:
	\begin{align*}
	2\left|\left<\mathcal{C}^T b,P_n\phi\right>_{L^2_P}\right|
	\leq
	2\left\|\mathcal{C}^T b\right\|_{L^2_P}\left\|P_n\phi\right\|_{L^2_P}
	\leq
	\frac{\left\|\mathcal{C}^T b\right\|^2_{L^2_P}}{\varepsilon^2}+\varepsilon^2\left\|P_n\phi\right\|^2_{L^2_P}.
	\end{align*}
Letting $\varepsilon^2=T$, we have
	\begin{align*}
	\left<\mathcal{M}_n^T \left(\begin{array}{c}\phi\\0\\b\end{array}\right),\left(\begin{array}{c}\phi\\0\\b\end{array}\right)\right>_{L^2_P\times L^2_P\times\R}
	\leq&\\
	-\left<\mathcal{A}_1^T P_n\phi,P_n\phi\right>_{L^2_P}
	+
	T{\left\|P_n\phi\right\|^2_{L^2_P}}
	&-
	P(T^{-2}-l^T)b^2
	+
	T^{-1}\left\|\mathcal{C}^T b\right\|^2_{L^2_P}.
	\end{align*}
By Lemma \ref{properties1}, $\mathcal{A}_1^T>\gamma>0$ for all $T$ sufficiently small, and therefore this expression is negative for all $\phi\in H^2_{P,0}$ and $b\in\R$, since $l^T$ and $\mathcal{C}^T$ are both bounded. Therefore, there exists  a $T_*>0$ such that for every $T<T_*$ there exists an $n+1$ dimensional subspace and on which $\mathcal{M}_n^T$ is negative definite. We conclude that
	\begin{equation}\label{eq:neg-small-t}
	\operatorname{neg}\left(\mathcal{M}_n^T\right)=n+1,
	\hspace{.5cm}
	\text{for all }T<T_*.
	\end{equation}
Notice that $T_*$ does not depend upon $n$.
\end{proof}

\subsection{$T=+\infty$}
We diagonalize $\mathcal{M}^\infty_n$ and count its negative eigenvalues. Considering \eqref{eq:mn-infty}, we see that it may be rewritten as
	\begin{equation}
	\mathcal{F}_n^\infty
	=
	\left(\begin{array}{ccc}
	{\mathcal{K}_{n}^\infty}&{0}&{0}\\{0}&{-\mathcal{A}_{1,n}^\infty}&{0}\\{0}&{0}&{Pl^\infty}
	\end{array}\right)
	\end{equation}
where ${\mathcal{K}_{n}^\infty}=\mathcal{A}_{2,n}^\infty+\left(\mathcal{B}_{n}^\infty\right)^*\left(\mathcal{A}_{1,n}^\infty\right)^{-1}\mathcal{B}_{n}^\infty$. This inversion is allowed since $\mathcal{A}_{1}^\infty$ is invertible on the image of $\mathcal{B}^\infty$, see Lemma \ref{well-defined}. We can therefore conclude that
	\begin{equation}\label{diag-l0}
	\begin{split}
	\operatorname{neg}\left(\mathcal{M}_n^\infty\right)
	&=\operatorname{neg}\left(\mathcal{A}_{2,n}^\infty+\left(\mathcal{B}_{n}^\infty\right)^*\left(\mathcal{A}_{1,n}^\infty\right)^{-1}\mathcal{B}_{n}^\infty\right)+\operatorname{neg}\left(-\mathcal{A}_{1,n}^\infty\right)+\operatorname{neg}\left(l^\infty\right)\\
	&=
	\operatorname{neg}\left(\mathcal{A}_{2,n}^\infty+\left(\mathcal{B}_{n}^\infty\right)^*\left(\mathcal{A}_{1,n}^\infty\right)^{-1}\mathcal{B}_{n}^\infty\right)+n-\dim \ker\left(\mathcal{A}_{1,n}^\infty\right)-\operatorname{neg}\left(\mathcal{A}_{1,n}^\infty\right)+\operatorname{neg}\left(l^\infty\right).
	\end{split}
	\end{equation}

Since there are only finitely may negative eigenvalues, we have the simple statement whose proof is omitted:
\begin{lemma}\label{limit-lemma}
There exists $N>0$ such that for all $n>N$ it holds that \[\operatorname{neg}\left(\mathcal{A}_{1,n}^\infty\right)=\operatorname{neg}\left(\mathcal{A}_{1}^\infty\right)\] and \[\operatorname{neg}\left(\mathcal{A}_{2,n}^\infty+\left(\mathcal{B}_{n}^\infty\right)^*\left(\mathcal{A}_{1,n}^\infty\right)^{-1}\mathcal{B}^\infty_n\right)=\operatorname{neg}\left(\mathcal{A}_{2}^\infty+\left(\mathcal{B}^\infty\right)^*\left(\mathcal{A}_{1}^\infty\right)^{-1}\mathcal{B}^\infty\right).\]
\end{lemma}

\subsection{$T$ large}
The case of $T\to+\infty$ is more difficult, due to the mere strong convergence $Q^T_\pm\to Q^\infty_\pm$. 
\begin{lemma}\label{small-la}
There exists $T^*>0$ such that for any $n\in\mathbb{N}$ and any $T>T^*$, $\operatorname{neg}(\mathcal{M}^T_n)\geq \operatorname{neg}(\mathcal{M}^\infty_n)$.
\end{lemma}
It is well-known that the spectra of a sequence of strongly continuous operators may have discontinuities (as a set). One can think of the following simple example: let $u_n$ be some orthonormal basis, and let $\pi_N$ be the orthogonal projection operator onto $\operatorname{span}(\{u_n\}_{n\geq N})$. Then $\pi_N\to0$ strongly as $N\to\infty$, but the spectrum of $\pi_N$ is $\{0,1\}$ for all $N$.
The proof of Lemma \ref{small-la} relies on the following theorem regarding the spectra of operators that are strongly continuous:
\begin{theorem}\label{thm:resolvent}
Let $A^T=-\Delta+V^T$ be a Schr\"odinger operator on $L^2(\mathbb{T}^d)$, with domain $H^2(\mathbb{T}^d)$ and with $\{V^T\}_{T\in(0,\infty]}$ a bounded family of strongly continuous, relatively compact perturbations of $\Delta$. Let $P_n:L^2(\mathbb{T}^d)\to L^2(\mathbb{T}^d)$ be the orthogonal projection operator onto the subspace associated to the first $n$ eigenvalues (counting multiplicity) of $A^\infty$. Define $A_n^T=P_nA^T P_n$. Let $\rho\in\rho(A^\infty)$ be an element in the resolvent set of the operator $A^\infty$. Then there exist $N=N(\rho)>0$ and $T^*=T^*(\rho)>0$ such that $\rho\in\rho(A_n^T)$ for all $n>N$ and for all $T>T^*$.
\end{theorem}

\begin{proof}[Sketch of proof]
The proof is by contradiction, showing that nontrivial solutions of $A_n^T f_n^T=\rho f_n^T$ for arbitrarily large $n$ and $T$ lead to a nontrivial solution of $A^\infty f=\rho f$, in contradiction to the assumption that $\rho$ is not an eigenvalue. The main difficulty is in understanding the convergence of terms of the form $P_n\Delta P_n f_n^T$. It can be shown that $P_n\Delta P_n f_n^T\to \Delta f$ in the $H^{-1}$ sense as $T,n\to\infty$. However bootstrapping this convergence to $L^2$ requires a more delicate analysis.
\end{proof}
The full proof can be found in \cite{Ben-Artzi2011b}. This theorem allows us to conclude that no eigenvalues of $\mathcal{M}_n^T$ cross through $0$ for large values of $n$ and $T$ if $0$ is in the resolvent set on $\mathcal{M}^\infty$. Lemma \ref{small-la} follows immediately.

\subsection{Conclusion}
Combining Lemmas \ref{limit-lemma} and \ref{small-la}, together with \eqref{eq:neg-small-t} and the condition \eqref{eq:thm1-condition} we conclude that for each $n$ sufficiently large, there exists $T_n\in(T_*,T^*)$ such that $\ker\mathcal{M}_n^{T_n}\neq\emptyset$. Now one only needs to let $n\to\infty$. This requires an argument analogous to the one provided in Theorem \ref{thm:resolvent}, and one can conclude that there indeed exists $T_0\in(T_*,T^*)$ such that $\ker\mathcal{M}^{T_0}\neq\emptyset$ and that there exists an element in this kernel which is a nontrivial solution of the system \eqref{eq:rvm-1.5}.

\section{On uniform ergodic theorems}\label{sec:ergodic}
The most difficult step in the proof of Theorem \ref{thm:main-intro} is hidden in Theorem \ref{thm:resolvent}. As mentioned above, this difficulty stems from the mere strong convergence of the ergodic averaging operators $Q_\pm^T$. Therefore here we take a step back and ask when such operators may have a limit also in the \emph{uniform} operator topology.

We start with the $L^2$ theory:

\begin{proof}[Proof of Theorem \ref{thm:unif-erg-l2}]
We follow the ideas set forth by von Neumann in his proof of the ergodic theorem \cite{VonNeumann1932a}, using the added structure that differential operators have via the Fourier transform. If $\{E(\lambda)\}_{\lambda\in\mathbb{R}}$ is the spectral family of the self-adjoint operator $H=-i\frac{d}{dx}:H^1(\mathbb{R})\subset L^2(\R)\to L^2(\R)$ we have the expression
	\begin{equation}\label{eq:dens-states1}
	\left(E(\lambda)f,g\right)_{L^2(\R)}
	=
	\int_{\xi\leq\lambda}\widehat{f}(\xi)\overline{\widehat{g}(\xi)}\;d\xi
	\end{equation}
where $\widehat{f}(\xi)=(2\pi)^{-1/2}\int_{\R}f(x)e^{-ix\xi}\ dx$ is the Fourier transform of $f$.
Whenever this expression is differentiable with respect to $\lambda$, we get
	\begin{equation}\label{eq:dens-states2}
	\frac{d}{d\lambda}\Big|_{\lambda=\lambda_0}\left(E(\lambda)f,g\right)_{L^2(\R)}
	=
	\widehat{f}(\lambda_0)\overline{\widehat{g}(\lambda_0)}.
	\end{equation}
The pointwise evaluations on the right hand side require $\widehat{f},\widehat{g}\in H^\sigma(\R)$ with $\sigma>1/2$ due to Sobolev embedding. This means that $f,g\in L^{2,\sigma}(\R)$ (see the definition in \eqref{eq:l2s}). We can therefore estimate
	\begin{equation}\label{eq:dens-states3}
	\left|\frac{d}{d\lambda}\Big|_{\lambda=\lambda_0}\left(E(\lambda)f,g\right)_{L^2(\R)}\right|
	\leq
	C(\sigma)\|f\|_{L^{2,\sigma}(\R)}\|g\|_{L^{2,\sigma}(\R)}
	\end{equation}
which implies that there exists an operator $A(\lambda_0):L^{2,\sigma}(\R)\to L^{2,-\sigma}(\R)$ such that $\left<A(\lambda_0)f,g\right>=\frac{d}{d\lambda}\big|_{\lambda=\lambda_0}\left(E(\lambda)f,g\right)_{L^2(\R)}$ where $\left<\cdot,\cdot\right>$ is the $(L^{2,-\sigma}(\R),L^{2,\sigma}(\R))$ dual space pairing (with respect to the $L^2(\R)$ inner product). Furthermore, the operator norm of $A$ may be estimated as
	\begin{equation}\label{eq:dens-states4}
	\left\|A(\lambda_0)\right\|_{\mathcal{B}(L^{2,\sigma}(\R),L^{2,-\sigma}(\R))}
	\leq
	C(\sigma).
	\end{equation}
We are now in a position to finally prove the theorem. Let $P^T=\frac{1}{2T}\int_{-T}^Te^{itH}dt$, then
	\begin{equation*}
	P^Tf
	=
	\frac{1}{2T}\int_{-T}^T\int_{\R}e^{it\lambda}dE(\lambda)f\ dt
	=
	\int_{\R}\frac{\sin\lambda T}{\lambda T}dE(\lambda)f
	=
	\int_{\R\setminus\{0\}}\frac{\sin\lambda T}{\lambda T}dE(\lambda)f
	\end{equation*}
where in the last equality we used the fact that $H$ has a trivial kernel, i.e. $E(\{0\})=0$. We estimate this integral by breaking it up into the following two integrals: $\int_{\R\setminus\{0\}}=\int_{I_\epsilon}+\int_{I_\epsilon^C}$ where $I_\epsilon=(-\epsilon,\epsilon)$ and $\epsilon>0$ (the first integral should be $\int_{I_\epsilon\setminus\{0\}}$, but since $E(\{0\})=0$ this does not matter). We start with the (simpler) integral $\int_{I_\epsilon^C}$:
	\begin{equation}\label{eq:ergodic-est1}
	\begin{split}
	\left\|\int_{I_\epsilon^C}\frac{\sin\lambda T}{\lambda T}dE(\lambda)f\right\|_{ L^{2,-\sigma}(\R)}^2
	&=
	\int_{I_\epsilon^C}\left|\frac{\sin{\lambda T}}{\lambda T}\right|^2d\left\|E(\lambda)f\right\|_{ L^2_{w}(\R)}^2\\
	&\leq
	\frac{1}{\epsilon^2 T^2}\int_{I_\epsilon^C}d\left\|E(\lambda)f\right\|_{ L^2(\R)}^2\\
	&\leq
	\frac{1}{\epsilon^2T^2}\int_{\R}d\left\|E(\lambda)f\right\|_{ L^2(\R)}^2\\
	&=
	\frac{1}{\epsilon^2T^2}\|f\|_{ L^2(\R)}^2\\
	&\leq
	\frac{1}{\epsilon^2T^2}\|f\|_{ L^{2,\sigma}(\R)}^2.
	\end{split}
	\end{equation}
For the other integral we need the estimate \eqref{eq:dens-states4}:
	\begin{equation}\label{eq:ergodic-est2}
	\begin{split}
	\left\|\int_{I_\epsilon}\frac{\sin\lambda T}{\lambda T}dE(\lambda)f\right\|_{ L^{2,-\sigma}(\R)}^2
	&=
	\left\|\int_{I_\epsilon}\frac{\sin\lambda T}{\lambda T}A(\lambda)f\ d\lambda\right\|_{ L^{2,-\sigma}(\R)}^2\\
	&\leq
	C(\sigma)\|f\|_{L^{2,\sigma}(\R)}^2\int_{I_\epsilon}\left|\frac{\sin{\lambda T}}{\lambda T}\right|^2d\lambda\\
	&\leq
	2\epsilon C(\sigma)\|f\|_{L^{2,\sigma}(\R)}^2.
	\end{split}
	\end{equation}
Combining the two estimates \eqref{eq:ergodic-est1} and  \eqref{eq:ergodic-est2} which hold for any $\epsilon>0$ we conclude that indeed $\lim_{T\to\infty}P^T=0$ in $\mathcal{B}(L^{2,\sigma}(\R),L^{2,-\sigma}(\R))$.
\end{proof}

Now we turn to the weighted-$L^2$ theory. In this case, determining the precise spectrum of the operator (and, indeed, even determining that the operator is self-adjoint) requires much more work. However, once this is done, the actual ergodic theorem is much simpler.

\begin{proof}[Proof of Theorem \ref{thm:self-adjoint}]
It is clear that $H_w$ is symmetric, closed and densely defined on $D^\alpha$. To show that it is essentially self-adjoint we let $g\in L^2_w(\R)$ and seek $h\in L^2_w(\R)$ such that
	\begin{equation}\label{eq:int-by-parts2}
	(H_wf,g)_{L^2_w(\R)}=(f,h)_{L^2_w(\R)},\quad\forall f\in D^\alpha.
	\end{equation}
By taking $f$ to be a smooth, compactly supported test function we can conclude that $g$ is differentiable and $-iw^{-1}\frac{d}{dx}g\in L^2_w(\R)$. However $C_0^\infty(\R)$ is not a core. Let $f\in C^\infty(\R)$ be such that $\lim_{x\to\infty}f(x)=\alpha\lim_{x\to-\infty}f(x)$. The left hand side of \eqref{eq:int-by-parts2} becomes
	\begin{align*}
	(H_wf,g)_{L^2_w(\R)}
	&=
	-i\int_{-\infty}^\infty\frac{d}{dx}f(x)\overline{g(x)}\ dx\\
	&=
	-i\lim_{R\to\infty}\int_{-R}^R\frac{d}{dx}f(x)\overline{g(x)}\ dx\\
	&=
	i\lim_{R\to\infty}\left[\int_{-R}^Rf(x)\frac{d}{dx}\overline{g(x)}\ dx-f(R)\overline{g(R)}+f(-R)\overline{g(-R)}\right].
	\end{align*}
Since $-iw^{-1}\frac{d}{dx}g\in L^2_w(\R)$ all limits exist so that we obtain
	\begin{align*}
	(H_wf,g)_{L^2_w(\R)}
	&=
	i\left[\int_{\R}f(x)\frac{d}{dx}\overline{g(x)}\ dx-f(\infty)\overline{g(\infty)}+f(-\infty)\overline{g(-\infty)}\right]\\
	&=
	i\left[\int_{\R}f(x)\frac{d}{dx}\overline{g(x)}\ dx-\alpha f(-\infty)\overline{g(\infty)}+f(-\infty)\overline{g(-\infty)}\right]\\
	&=
	i\int_{\R}f(x)\frac{d}{dx}\overline{g(x)}\ dx-if(-\infty)\left(\alpha\overline{g(\infty)}-\overline{g(-\infty)}\right)
	\end{align*}
which must equal the right hand side of \eqref{eq:int-by-parts2}:
	\[
	i\int_{\R}f(x)\frac{d}{dx}\overline{g(x)}\ dx-if(-\infty)\left(\alpha\overline{g(\infty)}-\overline{g(-\infty)}\right)
	=
	\int_\R f(x)\overline{h(x)}w(x)\ dx,\quad \forall f\in D^\alpha.
	\]
For this equality to hold in general, $g$ must satisfy $\overline\alpha{g(\infty)}=g(-\infty)$, which becomes ${g(\infty)}=\alpha g(-\infty)$ by multiplying by $\alpha$ and recalling that $|\alpha|=1$. Hence we conclude that $g\in D^\alpha$, and therefore $H_w$ is essentially self-adjoint on $D^\alpha$.

Moreover, we can determine the spectrum of $H_w^\alpha$ by looking for solutions of $H_w^\alpha f=\lambda f$. Such solutions have the form
	\[
	f(x)
	=
	Ce^{i\lambda\int_0^xw(t)dt}.
	\]
The condition $f(\infty)=\alpha f(-\infty)$ becomes (letting $\alpha=e^{i\beta},\ \beta\in[0,2\pi)$)
	\[
	\lambda\int_0^\infty w(t)\ dt
	=
	\beta+\lambda\int_0^{-\infty}w(t)\ dt+2\pi k,\quad k\in\Z
	\]
so that we conclude
	\begin{equation}\label{eq:1d-weighted-eigenvalues}
	\lambda_k^\beta
	=
	\|w\|_{L^1(\R)}^{-1}(\beta+2\pi k),\quad k\in\Z.
	\end{equation}

The fact that there are no additional points in the spectrum is due to $H_w^\alpha$ having compact resolvent. Indeed, let us show that $R^\alpha_w(z)=(H_w^\alpha-z)^{-1}$, where $z\in\C\setminus\Sigma(H_w^\alpha)$, is a compact operator $ L^2_w(\R)\to D^\alpha\subset L^2_w(\R)$. It suffices to show that the embedding $ D^\alpha\subset L^2_w(\R)$ is compact. Let $K\subset D^\alpha$ be a bounded set. All elements of $K$ are uniformly bounded near $\pm\infty$, and therefore for every $\epsilon>0$ there exists $M>0$ such that $\int_{|x|>M}|f(x)|^2w(x)\ dx<\epsilon$ for every $f\in K$. Concluding that $K$ is compact in $ L^2_w(\R)$ is standard, using Rellich's theorem on $|x|<M$ and the smallness of the tails on $|x|>M$.
\end{proof}

Proving Theorem \ref{thm:unif-erg-wl2} is now simple due to the existence of a spectral gap:
\begin{proof}[Proof of Theorem \ref{thm:unif-erg-wl2}]
Let $\{E(\lambda)\}_{\lambda\in\R}$ be the spectral family of $H_w^\alpha$ and let $P^T=\frac{1}{2T}\int_{-T}^Te^{itH_w^\alpha}dt$. Then $P=E(\{0\})$ is the orthogonal projection onto the kernel of $H_w^\alpha$. Hence, as before, we can show that we have the representation
	\begin{equation*}
	(P^T-P)f
	=
	\int_{\R\setminus\{0\}}\frac{\sin\lambda T}{\lambda T}dE(\lambda)f
	\end{equation*}
which we again break up into integrals over $I_\epsilon\setminus\{0\}=(-\epsilon,\epsilon)\setminus\{0\}$ and $I_\epsilon^C$. If $\epsilon>0$ is sufficiently small, the first integral makes no contribution due to the spectral gap. The second integral is treated as in \eqref{eq:ergodic-est1}, where the exact properties of the spectral family (and, in particular, whether the spectral measure is absolutely continuous or has atoms) do not matter. Moreover, we observe that these arguments do not require $f$ to be in any special subspace of $L^2_w(\R)$ as was the case before.
\end{proof}

\section{Properties of the operators}\label{sec:prop-oper}
Here we collect all the important properties of the operators appearing in Section \ref{sec:self-adjoint}. The proofs are technical and we refer to \cite{Ben-Artzi2011} for the details.

\begin{lemma}[Properties of $D^\pm$]
$D^\pm$ are skew-adjoint operators on $L_\pm^2$. Their null spaces $\ker{D^\pm}$ consist of all functions  in $L_\pm^2$ that are constant on each connected component in $\R\times\R^2$ of $\{e^\pm=const\text{ and }p^\pm=const\}$. In particular, $\ker{D^\pm}$ contain all functions of $e^\pm$ and of $p^\pm$.
\end{lemma}

\begin{lemma}[Properties of $ Q^T_\pm$]\label{propql}
Let $0<{T}<\infty$.
	\begin{enumerate}
	\item\label{qlnorm}
		$ Q^T_\pm$ map $L_\pm^2\to L_\pm^2$ with operator norm = 1. 
		\item\label{laz}
For all $m\in L_\pm^2$, $\left\| Q^T_\pm m-Q^\infty_\pm m\right\|_\pm\to0$ as ${T}\to\infty$.
	\item\label{lainfty}
	For all $m\in L_\pm^2$, $\left\| Q^T_\pm m-m\right\|_\pm\to0$ as ${T}\to0$.

	\item\label{lanotz}
	If $S>0$, then $\left\| Q^T_\pm-Q^S_\pm\right\|=O(|T-S|)$ as $T\to S$, where $\|\cdot\|$ is the operator norm from $L_\pm^2$ to $L_\pm^2$.	
	
	\item
	The projection operators $Q^\infty_\pm$ preserve parity with respect to the variable $v_1$.
	
	\end{enumerate}
\end{lemma}

\begin{lemma}[Properties of $\mathcal{A}_1^T, \mathcal{A}_2^T$]\label{properties1}
Let $0\leq{T}<\infty$.
	\begin{enumerate}
	\item\label{op-norm}
	$\mathcal{A}_1^T$ is self-adjoint on $ L_{P,0}^2 $. $\mathcal{A}_2^T$ is self-adjoint on ${L_P^2}$. Their domains are ${H_{P,0}^2}$ and ${H_P^2}$, respectively, and their spectra are discrete.
	\item
	For all $h\in{H_{P,0}^2}$, $\|\mathcal{A}_1^T h-\mathcal{A}_1^\infty h\|_{{L_P^2}}\to0$ as ${T}\to\infty$. The same is true for $\mathcal{A}_2^T$ with $h\in{H_P^2}$.
	\item
	For $i=1,2$ and $S>0$, it holds that $\|\mathcal{A}^T_i-\mathcal{A}^S\|=O(|{T}-S|)$ as ${T}\to S$, where $\|\cdot\|$ is the operator 	norm from ${H_{P,0}^2}$ to ${L_P^2}$ in the case $i=1$, and from ${H_P^2}$ to ${L_P^2}$ in the case $i=2$.
	\item
	For all $h\in{H_{P,0}^2}$, $\|\mathcal{A}_1^T h+\partial_x^2h\|_{{L_P^2}}\to0$ as ${T}\to0$.
	\item
	When thought of as acting on ${H_P^2}$ (rather than ${H_{P,0}^2}$), the null spaces of $\mathcal{A}_1^T$ and $\mathcal{A}_1^\infty$ both contain the constant functions.
	
	\item\label{operators-positive}
	There exist constants $\gamma>0$ and $\overline{T}>0$ such that $\mathcal{A}^T_i>\gamma>0$ for all ${T}\leq\overline{T}$ and $i=1,2$.
	\end{enumerate}
\end{lemma}

\begin{lemma}[Properties of $\mathcal{B}^T, \mathcal{C}^T, \mathcal{D}^T$]\label{properties2}
Let $0<{T}<\infty$.

	\begin{enumerate}
	\item\label{blnorm}
	$\mathcal{B}^T$ maps ${L_P^2}\to{L_P^2}$ with operator bound independent of ${T}$. Moreover, $\mathrm{Ran}(\mathcal{B}^\infty)\subset\{1\}^\perp$.
	
	\item\label{bz}
	For all $h\in{L_P^2}$, as ${T}\to\infty$ we have: $\|\mathcal{B}^T h-\mathcal{B}^\infty h\|_{{L_P^2}}\to0$ and $\|\mathcal{C}^T h\|_{{L_P^2}},\|\mathcal{D}^T h\|_{{L_P^2}}\to0$.

	\item
	If $S>0$, then $\|\mathcal{B}^T-\mathcal{B}^S\|=O(|{T}-{S}|)$ as ${T}\to{S}$, where $\|\cdot\|$ is the operator norm from ${L_P^2}$ to ${L_P^2}$. The same is true for $\mathcal{C}^T, \mathcal{D}^T$.

	\item\label{bl-infty}
	For all $h\in{L_P^2}$, $\|\mathcal{B}^T h\|_{{L_P^2}}\to0$ as ${T}\to0$. The same is true for $\mathcal{C}^T, \mathcal{D}^T$.
	
	\end{enumerate}
\end{lemma}

\begin{lemma}[Properties of $l^{T}$]\label{propll}
Let $0<{T}<\infty$.
	\begin{enumerate}
	\item
	$l^{T}\to l^\infty$ as ${T}\to\infty$.
	
	\item
	$l^{T}$ is uniformly bounded in ${T}$.
	
	\end{enumerate}
\end{lemma}

\begin{lemma}[Properties of $\mathcal{M}^T$]\label{propml}
To simplify notation, we write $u$ for a generic element $\left(\phi,\psi,b\right)\in{H_P^2}\times{H_P^2}\times\R$.
\begin{enumerate}
\item
For all ${T}\geq0$, $\mathcal{M}^T$ is self-adjoint on ${L_P^2}\times{L_P^2}\times\R$ with domain ${H_P^2}\times{H_P^2}\times\R$.

\item
For all $u\in{H_P^2}\times{H_P^2}\times\R$, $\|\mathcal{M}^T u-\mathcal{M}^\infty u\|_{{L_P^2}\times{L_P^2}\times{L_P^2}}\to0$ as ${T}\to\infty$.

\item
If $S>0$, then $\|\mathcal{M}^T-\mathcal{M}^S\|\to0$ as $T\to S$, where $\|\cdot\|$ is the operator norm from ${H_{P,0}^2}\times{H_P^2}\times\R$ to ${L_P^2}\times{L_P^2}\times{L_P^2}$.

\end{enumerate}
\end{lemma}

\begin{lemma}\label{well-defined}
The operator $\left(\mathcal{B}^\infty\right)^*\left(\mathcal{A}_1^\infty\right)^{-1}\mathcal{B}^\infty$ is a well-defined bounded operator from ${L_P^2}\to{L_P^2}$.
\end{lemma}

\bibliography{library}
\bibliographystyle{amsplain}

\end{document}